\documentclass[12pt, reqno]{amsart}
\usepackage{amssymb, amsmath, amsthm}
\usepackage[backref]{hyperref}
\usepackage{amscd}   
\usepackage{fullpage}
\usepackage[all]{xy} 

\DeclareFontEncoding{OT2}{}{} 

\usepackage[usenames,dvipsnames]{color}

\newtheorem{lemma}{Lemma}[section]
\newtheorem{theorem}[lemma]{Theorem}
\newtheorem{proposition}[lemma]{Proposition}
\newtheorem{prop}[lemma]{Proposition}
\newtheorem{cor}[lemma]{Corollary}

\newtheorem{claim*}{Claim}
\newtheorem{thm}[lemma]{Theorem}
\newtheorem{defn}[lemma]{Definition}

\theoremstyle{definition}
\newtheorem{remark}[lemma]{Remark}

\newtheorem{rmk}[lemma]{Remark}

\numberwithin{equation}{section}
\numberwithin{table}{section}

\title{Abelian Calabi-Yau threefolds: N\'eron models and rational points}
\author{Fedor Bogomolov}
\author{Lars Halvard Halle}
\author{Fabien Pazuki}
\author{Sho Tanimoto}

\address{Courant Institute New York University, New York, NY 10012, USA }
\address{National Research University Higher School of Economics, Russian Federation}
\email{bogomolov@cims.nyu.edu}
\address{Department of Mathematical Sciences, University of Copenhagen, Universitetsparken 5 2100 Copenhagen $\emptyset$, Denmark}
\email{larshhal@math.ku.dk}
\email{fpazuki@math.ku.dk}
\email{sho@math.ku.dk}
\urladdr{https://shotanimoto.wordpress.com}

\date{}
\subjclass[]{}

\begin{document}


	\maketitle
	

\section{Introduction}

Let $k$ be a number field and let $X$ be a projective variety defined over $k$. Lang conjectured that when $X$ is of general type, for any finite extension $k'/k$, the set of $k'$-rational points $X(k')$ is not Zariski dense in $X$. One may ask for the converse: which varieties $X$ satisfy \emph{potential density} in the Zariski topology, \textit{i.e.}, $X(k')$ is Zariski dense in $X$ for some finite extension $k'/k$? In view of this conjecture of Lang, one is lead to ask whether rational points are potentially dense if neither $X$ nor its unramified coverings admit a morphism onto a variety of general type. It is expected that the class of Fano varieties satisfies this potential density property, and it is verified for most Fano 3-folds, except for double covers of $\mathbb P^3$ ramified along smooth surfaces of degree $6$ where the result is still unknown. See \cite{HaTs00} and \cite{BoTs99} for more details.

An interesting question is: how about intermediate classes of varieties, \textit{e.g.}, varieties of Kodaira dimension zero? In dimension $1$, a curve of Kodaira dimension zero is an elliptic curve and it is well-known that it satisfies potential density. In dimension $2$, according to the classification theory, there are $4$ types of minimal smooth projective surfaces of Kodaira dimension $0$:
\begin{itemize}
\item abelian surfaces;
\item bielliptic surfaces;
\item Enriques surfaces;
\item K3 surfaces.
\end{itemize}
Abelian surfaces satisfy potential density (see \cite[Proposition 3.1]{HaTs00}) so bielliptic surfaces do as well. Potential density of Enriques surfaces is proved in \cite{BoTs98}, hence the only remaining case is the class of K3 surfaces. It is shown in \cite{BoTs00} that if a K3 surface $X$ admits an infinite group of automorphisms, or an elliptic fibration, then $X$ satisfies potential density. 
\\

A problem which will be addressed in this paper is a higher dimensional analogue of the aforementioned theorems, using abelian surface fibrations instead of elliptic fibrations. The question of whether it could be possible to find Calabi-Yau threefolds where the rational points are potentially dense was asked in Tschinkel's ICM talk \cite{Tsc06} (more precisely, it is the question raised just after his Problem 3.5).

We will thus focus on Calabi-Yau threefolds with an abelian surface fibration, and explore the interplay between Zariski potential density properties of their $k$-rational points, the structure of the Mordell-Weil group of the generic fibers, and their N\'eron models. These notions are intimately linked. For instance, while solving the question of whether the $k$-rational points are (potentially) Zariski dense on a Horrocks-Mumford quintic threefold in Section~\ref{Horrocks}, we first compute the structure of the Mordell-Weil group of its generic fiber, and computing this structure requires in our proof a fine property of N\'eron models. As we believe that the description of the N\'eron models we use will be of interest independently, and concerns more general cases, we give it here as a first main theorem. 
\begin{thm}
Let $k$ be a field of characteristic zero
and $X$ a Calabi-Yau threefold admitting an abelian surface fibration $f \colon X \rightarrow \mathbb P^1$ with a section.
Then the $f$-smooth part $\mathcal A \subset X \rightarrow \mathbb P^1$ forms the N\'eron-model of the generic fiber $X_\eta$.
\end{thm}
This result is a consequence of the more general Theorem \ref{theorem: Neronmodel}; its statement and proof is given in Section~\ref{Neron section}. See Definition~\ref{defn: neron} for the definition of N\'eron models.
In Section~\ref{sec: potential} we discuss potential density of abelian Calabi-Yau threefolds and show that if an abelian Calabi-Yau threefold admits a second fibration, then it satisfies potential density in the Zariski topology.

\begin{thm}(Theorem \ref{thm: doublefibration})
Let $X$ be a Calabi-Yau threefold defined over a number field $k$ admitting an abelian surface fibration $f \colon X \rightarrow \mathbb P^1$ with a section.
Suppose that $X$ has an effective movable non-big divisor $D$ such that the abelian fibration on $X$ does not factor through the Iitaka fibration of $D$ birationally. Then $X$ satisfies the potential density property with respect to the Zariski topology.
\end{thm}

\begin{rmk}
The existence of an effective non-big movable divisor implies that after some birational modification, $X$ admits another fibration. So in this sense, $X$ admits double fibrations.
\end{rmk}
The idea of using double fibrations to study potential density dates back at least to Swinnerton-Dyer \cite{Sw68}, see also more recently \cite{LMvL10} and \cite{Sw13}, all of which concern \textit{surfaces}.
 In sections \ref{(1,4)}, \ref{(1,6)}, \ref{Pfaffian}, \ref{2222}, \ref{Grassmannians} we will focus on examples of abelian Calabi-Yau \textit{threefolds} described in \cite{GrPo01} as explicit applications of our Theorem \ref{thm: doublefibration}. More precisely, we study Calabi-Yau threefolds fibered by polarized abelian surfaces of type $(1, d)$ where $d=4, 5, 6, 7, 8, 10$ and their arithmetic properties, \textit{e.g.} Zariski density and Mordell-Weil groups. It provides, together with the Horrocks-Mumford case in Section \ref{Horrocks}, several explicit examples of abelian Calabi-Yau threefolds satisfying potential density, answering Tchinkel's original question with two different methods.
 \\
	
\noindent
{\bf Acknowledgements}
The authors would like to thank Brendan Hassett for useful discussions and invaluable inputs.
The first author was partially supported by the Russian Academic Excellence Project `5-100' and by Simons Travel Grant.
The third author is supported by ANR-14-CE25-0015 Gardio.
The last three authors are supported by Lars Hesselholt's Niels Bohr Professorship.
	
\section{Generalities}
\label{sec: general}

\subsection{Calabi-Yau threefolds}

\begin{defn}
Let $k$ be a field of characteristic zero.
 {\it A Calabi-Yau threefold} is a geometrically connected smooth projective threefold $X$ satisfying $\omega_X \cong \mathcal O_X$ and $h^1(\mathcal O_X) =0.$
 We say that $X$ is a Calabi-Yau threefold in the strict sense if $X$ is additionally simply-connected.
\end{defn}

\begin{defn}
Let $C$ be a smooth curve of genus zero defined over $k$.
 A Calabi-Yau threefold $X$ is called {\it an abelian Calabi-Yau threefold} if it admits a fibration $\pi \colon X \rightarrow C$ defined over $k$
 whose generic fiber $X_\eta$ is an abelian surface with an identity over the function field $k(\eta)$. It is called {\it simple}
 if $X_\eta$ is a simple abelian surface over $k(\eta)$.
\end{defn}

\begin{rmk}
\label{rmk: P^1}
Over an algebraically closed field of characteristic zero, if we have an abelian fibration on a Calabi-Yau threefold $X$, then the base must be $\mathbb P^1$.
See \cite[Lemma 1.2]{GrPo01}.
\end{rmk}

\subsection{Abelian varieties}

\subsubsection{Polarization}

We recall the definition of a polarized abelian variety.

\begin{defn}
An abelian variety over $\mathbb{C}$ is a complex torus admitting a positive definite line bundle. A polarization on a complex torus $X$ is the first Chern class $H=c_1(L)$ of a positive definite line bundle $L$ on $X$. The pair $(X,H)$ is then a polarized abelian variety. 
\end{defn}

An ample line bundle $L$ on $A$ induces a map $\phi_{L}$ from $A$ to its dual $\check{A}$ given by $\phi_{L}(x)=t_x^{*}L\otimes L^{-1}$, where $t_x$ is the translation by $x$ morphism. The kernel $K(L)$ of $\phi_L$ is of the form $K(L)\simeq (\mathbb{Z}/d_1\mathbb{Z}\oplus\cdots \oplus \mathbb{Z}/d_g\mathbb{Z})^{\oplus 2}$, where $d_1\vert d_2\vert\ldots\vert d_g$ are positive integers. This decomposition only depends on $H$. The ordered $g$-tuple $(d_1,\ldots, d_g)$ is called the type of the polarization.  A polarization of type $(1,\ldots, 1)$ is called a principal polarization. For further details, see for instance \cite{GrPo98} page 335.

\subsubsection{Finiteness statements}

We recall in this paragraph two classical finiteness theorems, the Mordell-Weil Theorem for abelian varieties over number fields and the Lang-N\'eron Theorem for abelian varieties over more general fields, including function fields of characteristic zero.

\begin{theorem}\label{Mordell-Weil}(Mordell-Weil)
Let $A$ be an abelian variety defined over a number field $k$. Let $A(k)$ denote the set of $k$-rational points of $A$. Then there exist a non-negative integer $r_k$ and a finite subgroup $A(k)_{tors}\subset A(k)$ such that $$A(k)\simeq \mathbb{Z}^{r_k}\times A(k)_{tors}.$$ The integer $r_k$ is called the Mordell-Weil rank of $A(k)$.
\end{theorem}

\begin{remark}\label{simplicity}
If $A$ is a simple abelian variety defined over a number field $k$, then $A(k)$ is Zariski dense in $A$ if and only if $r_k\geq 1$.
\end{remark}

\begin{theorem}(Lang-N\'eron)\label{Lang-N}
Let $k$ be a field. If $K/k$ is a finitely generated regular extension and $A$ is an abelian variety over $K$, then $A(K)/\mathrm{Tr}_{K/k}(A)(k)$ is a finitely generated group.
\end{theorem}
\begin{proof}
The proof is given in \cite{LaNe59}. A modern exposition of the proof is given in \cite{Con06}. A regular extension is a separable extension $K/k$ with $k$ algebraically closed in $K$.
\end{proof}

\begin{remark}\label{simplicity function fields}
If $k$ is a number field, then $\mathrm{Tr}_{K/k}(A)(k)$ is finitely generated by the Mordell-Weil Theorem \ref{Mordell-Weil}, so $A(K)$ will also be finitely generated by applying Lang-N\'eron over the algebraic closure $\bar{k}$ of $k$ viewed as constant field in $\bar{k}K$, and by remarking that $\mathrm{Tr}_{\bar{k}K/\bar{k}}(A)(\bar{k})=\mathrm{Tr}_{K/k}(A)(\bar{k})$ and that $A(K)/\mathrm{Tr}_{K/k}(A)(k)\to A(\bar{k}K)/\mathrm{Tr}_{K/k}(A)(\bar{k})$ is injective. The details are given in \cite{Con06}. If in turn $A/K$ is simple, then $A(K)$ is Zariski dense in $A$ if and only if the rank of $A$ over $K$ is positive, as in Remark \ref{simplicity}.
\end{remark}

\subsection{Flops}

We recall results in \cite{Kaw88} and \cite{Kol89} stating that
every pair of birational 3-dimensional minimal models can be connected by birational operations so-called {\bf flops}.
Let $f \colon X\rightarrow Y$ be a morphism of projective varieties.
We denote the space of $1$-cycles modulo numerical equivalence by $N_1(X)$.
Let $N_1(X/Y)$ be the space generated by curves contracted by $f$.
\begin{defn}
Let $X$ be a normal projective variety.
An extremal flopping contraction is a birational morphism $f \colon  X \rightarrow Y$ to a normal projective variety
such that
\begin{itemize}
\item $f$ is small, i.e., the exceptional locus $\mathrm{Ex}(f)$ has codimension $\geq 2$;
\item $K_X$ is $f$-trivial;
\item $N_1(X/Y)$ is one dimensional.
\end{itemize}
Let $D$ be a $\mathbb Q$-Cartier divisor such that $-D$ is $f$-ample.
Then the $D$-flop is an extremal flopping contraction $f^+ \colon X^+ \rightarrow Y$
such that the strict transform $D^+$ of $D$ is $f$-ample.
The $D$-flop is unique and does not depend on the choice of $D$.
The $D$-flop exists when $X$ has only terminal singularities.
This is a theorem in dimension $3$ by Kawamata \cite{Kaw88} and in higher dimension when $D$ is big by \cite{BCHM}.
\end{defn}

\begin{lemma}{\cite[Lemma 4.3, Proposition 4.6, and Corollary 4.11]{Kol89}}
\label{lemma:flops}
Let $X$ and $X'$ be projective threefolds with $\mathbb Q$-factorial terminal singularities 
such that $K_X$ and $K_{X'}$ are nef.
Suppose that we have a birational map $f \colon X \dashrightarrow X'$.
Then $f$ is an isomorphism in codimension $1$.
Moreover, the indeterminacy of $f$ is a union of rational curves,
and $X$ and $X'$ have the same analytic singularities.
\end{lemma}

Let $X$ be an abelian Calabi-Yau threefold with an abelian surface fibration 
$\pi \colon X \rightarrow S$, where $S$ is a smooth projective curve.
Pick a section $P \in X(S)$. Then translation by $P$ defines a birational map
$f_P \colon  X \dashrightarrow X$.
It follows from Theorem~\ref{theorem: Neronmodel} that the birational map $f$ is well-defined
on the $\pi$-smooth locus $\mathrm{Sm}(X)$.
In particular, the indeterminacies of $f$ are in fibers of $f$. 
This leads to the following theorem.

\begin{theorem}{\cite[Theorem 4.9]{Kol89}}
Suppose that $X$ is an abelian Calabi-Yau threefold
with an abelian surface fibration $\pi \colon X \rightarrow S$.
Fix a section $P \in X(S)$. Consider the translation $f_P$ by $P$.
Let $D'$ be an ample divisor on $X$ and $D = (f_P^{-1})_*(D')$ the strict transform of $D$.

By running the $K_X + \epsilon D$-minimal model program for sufficiently small $\epsilon > 0$, we can decompose $f_P$ into a finite sequence of $D$-flops.
Moreover this MMP is a relative MMP with respect to the fibration $\pi \colon X \rightarrow S$.
\end{theorem}

\begin{cor}{\cite[Theorem 1.2]{Ogu14}}
\label{cor: oguiso}
Let $X$ be an abelian Calabi-Yau threefold
with an abelian surface fibration $f \colon X \rightarrow S$.
Suppose that the Picard rank of $X$ is two.
Then the Mordell-Weil group $\mathrm{MW}(A_\eta) = X(S)$ is a finite group
where $A_\eta$ denotes the generic fiber.
\end{cor}
\begin{proof}
Since the Picard rank is two, there is no flop $\sigma \colon X \dashrightarrow X'$ such that the strict transform of a general fiber $A_s$ is nef on $X'$ as well.
Thus, we have $\mathrm{MW}(A_\eta) \subset \mathrm{Aut}(X)$, where $\mathrm{Aut}(X)$ is the group of automorphisms of $X$. It follows from \cite[Theorem 1.2, case (1)]{Ogu14} that $\mathrm{Aut}(X)$ is finite.
\end{proof}

\section{N\'eron models of abelian Calabi-Yau varieties}\label{Neron section}
In this section, we let $k$ denote a field of characteristic $0$. We consider a flat morphism 
$$ f \colon X \to S $$ 
of smooth projective geometrically connected varieties over $k$. We moreover assume that $S$ is a curve, that the generic fiber of $f$ is an abelian variety over the function field of $S$, and that the canonical sheaf $\omega_X$ of $X$ is trivial. When these assumptions are met, we shall prove that the smooth locus of $f$ forms the \emph{N\'eron model} of its generic fiber.

We shall start by discussing some preliminary material concerning N\'eron models and Kulikov models. 
\subsection{N\'eron models} 
Let $S$ denote a connected Dedekind scheme with function field $K$, and let $A$ be an abelian variety over $K$.

\begin{defn}
\label{defn: neron}
The N\'eron model of $A$ is an $S$-group scheme $\mathcal{A}$ which is smooth, separated and of finite type such that the following universal property is true:
 \begin{center}
 For each smooth $S$-scheme $Y$ and each $K$-morphism $u_K : Y_K \rightarrow A$,
 there is a unique $S$-morphism $u: Y \rightarrow \mathcal{A}$ extending $u_K$.
 \end{center}
 This property is called {\bf the N\'eron mapping property}.
\end{defn}

For abelian varieties, the N\'eron model always exists (\cite[Theorem 1.4/3]{neron}). It is immediate from the N\'eron mapping property that $\mathcal{A}$ is unique up to canonical isomorphism, and that it carries a unique $S$-group scheme structure extending the group structure of $A$. Also, we have a canonical bijection $\mathcal{A}(S) = A(K)$.

\subsubsection{N\'eron models and descent}\label{subsubsec-descent}
Following \cite[Definition 3.6/1]{neron}, we say that a faithfully flat extension $ R \subset R' $ of discrete valuation rings has \emph{ramification index $1$}, if a uniformizer $\pi$ of $R$ induces a uniformizer of $ R' $, and if moreover the residue field extension is separable. The prime example to keep in mind, and the only case we shall use, is when $R'$ is the $\mathfrak{m}_R$-adic completion of $R$.

By \cite[Theorem 7.2/1]{neron}, the formation of N\'eron models descends along extensions of ramification index $1$. More precisely, let $A$ be an abelian variety defined over the fraction field $K$ of $R$. Let $A'$ denote the pullback of $A$ to the fraction field $K'$ of $R'$, and assume that $A'$ admits a N\'eron model $\mathcal{A}'$ over $R'$. Then $A$ admits a N\'eron model $\mathcal{A}$ over $R$, and, moreover, the natural basechange morphism
$$ \mathcal{A} \times_R R' \to \mathcal{A}' $$
extending the canonical isomorphism on the generic fibers is an isomorphism. Moreover, by faithful flat descent (cf.~e.g.~\cite[Proposition 6.2/D.4]{neron}), $ \mathcal{A} $ is unique up to canonical isomorphism.

\subsection{Kulikov models}\label{sec-Kulikov}
In this paragraph, we recall some terminology and a result from \cite[Section 6]{HaNi} which we will use to study N\'eron models for abelian Calabi-Yau varieties. 
We denote by $T$ a complete discrete valuation ring of equal characteristic zero, with quotient field $L$.

\begin{defn}\label{def-Kulikov}
Let $Y$ be a smooth projective geometrically connected $L$-variety, and assume that $\omega_{Y} = \mathcal{O}_Y $. A \emph{Kulikov model} of $Y$ is a regular, flat and proper $T$-model $\mathcal{Y}/T$ whose relative dualizing sheaf is trivial.
\end{defn}

In the context of abelian varieties, Kulikov models have the pleasant property that they form minimal compactifications of N\'eron models. The key result we shall use is \cite[Cor.~6.7]{HaNi}, which is stated as Proposition \ref{prop-Kulikov} below. 

\begin{prop}\label{prop-Kulikov}
If $A$ is an abelian $L$-variety, and $\mathcal{Y}/T$ is a Kulikov model of $A$, then the $T$-smooth locus $\mathrm{Sm}(\mathcal{Y}) $ of $\mathcal{Y}$ is a N\'eron model of $A$. 
\end{prop}
\begin{proof}
For the convenience of the reader, we include a brief sketch of the proof. Let $\mathcal{A}$ denote the N\'eron model of $A$. By the mapping property, there exists a unique morphism 
$$ h \colon \mathrm{Sm}(\mathcal{Y}) \to \mathcal{A} $$ 
extending the canonical isomorphism on the generic fibers. Since $\mathcal{Y}$ is regular, it is a standard fact that the $T$-smooth locus $\mathrm{Sm}(\mathcal{Y})$ forms a \emph{weak} N\'eron model of $A$. This means that $\mathrm{Sm}(\mathcal{Y})$ has the mapping property for \'etale points, and knowing this, it is straightforward to deduce that the morphism $h$ is surjective. 

The assumption that $\mathcal{Y}$ is Kulikov ensures that $h$ is in fact an open immersion. This is somewhat involved, and we refer to \cite[Section 6]{HaNi} for details. Now $h$ is a surjective open immersion, thus it is an isomorphism.
\end{proof}

\begin{remark}
Proposition \ref{prop-Kulikov} is classical if $A$ is an elliptic curve over $L$. In this case, the minimal regular model $\mathcal{E}/T$ forms a Kulikov model for $A$, and its $T$-smooth locus forms a N\'eron model (cf.~e.g. \cite[Theorems 9.4.35 and 10.2.14]{liu}). To be precise, one does not need to make any assumption on the residue characteristic of $T$. 

When $\mathrm{dim}(A) > 1$, the statement does not seem to have been known previous to \cite{HaNi}. 
\end{remark}

\begin{remark}
The existence of a Kulikov model for an abelian variety $A/L$ as above is not guaranteed. In fact, one can find examples of Jacobians of genus $2$ curves where no Kulikov model exists (see the discussion after \cite[Corollary 6.9]{HaNi}). 
\end{remark}

\subsection{}

We now return to our fibration
$$ f \colon X \to S. $$ 
Note that $f$ is both a projective morphism and a local complete intersection (cf. e.g.~\cite[Example 6.3.18]{liu}). Then, by \cite[Theorem 6.4.32]{liu}, the relative dualizing sheaf of $f$ coincides with the relative canonical sheaf $\omega_{X/S}$.

Let $ s \in S $ be an arbitrary closed point. We denote by $ R = \mathcal{O}_{S,s}$ the local ring at $s$, and by $R'$ the $\mathfrak{m}_R$-adic completion of $R$. We let $K$ and $K'$ denote the quotient fields of $R$ and $R'$, respectively.

Let $\mathcal{X}$ denote the pullback of $X$ via the canonical map $ \mathrm{Spec}(R) \to S $. Then $\mathcal{X}$ is a regular model of its generic fiber $ X_{K} $, and it is projective and flat over $R$. Pulling back further to $R'$ gives a flat and projective model $\mathcal{X}'/R'$. 

\begin{lemma}\label{lemma-canonical}
Let $\mathcal{X}'/R'$ be defined as above. Then $\mathcal{X}'$ is a Kulikov model. 

\end{lemma}
\begin{proof}
Let $ U \subset S $ be an open neighbourhood of $s$ such that $ \omega_{U} := \omega_{S} \vert_U $ is trivial. Let $ f \colon X_U \to U $ denote the restriction of $ f \colon X \to S $. Then we find that
$$f^* \omega_{U} \otimes \omega_{X_U / U} \cong \omega_{X_U},  $$
so $ \omega_{X_U / U} \cong \omega_{X_U} \cong \mathcal{O}_{X_U} $, since $ \omega_{X_U} $ is the restriction of the trivial line bundle $ \omega_{X} $ to the open subset $X_U$.

Since the formation of relative canonical sheaves commute with pullback, the sheaf $ \omega_{\mathcal{X}'/R'} $ is the pullbacks of $ \omega_{X_U / U} $ along the canonical projection $ \mathcal{X}' \to X_U $, and is therefore trivial as well (note that for this conclusion, we use the fact that $f$ is a flat, quasi-projective l.c.i.).

It remains to show that $\mathcal{X}'$ is regular. As the generic fiber is smooth, it suffices to consider (closed) points in the special fiber. However, the projection map $ \pi \colon \mathcal{X}' \to X $ restricts to an isomorphism on the special fibers over $s$. Moreover, for any closed point $ x \in X_s $, $\pi$ induces an isomorphism of the completions of the local rings $\mathcal{O}_{X,x}$ and $\mathcal{O}_{\mathcal{X}', x}$ (see the proof of \cite[Lemma 8.3.49]{liu}). Since $X$ is regular, this easily implies that also $\mathcal{X}'$ is regular, since this can be checked on the completed local rings at the closed points.

\end{proof}

\subsection{Conclusion}
We are now ready to prove our main result in this section.

\begin{theorem}
\label{theorem: Neronmodel}
The $f$-smooth locus $\mathcal{A}$ of $ f \colon X \to S $ is the N\'eron model of the generic fibre $ X_{K(S)} $. 
\end{theorem}
\begin{proof}
By \cite[Proposition 1.2/4]{neron}, it suffices, for every closed point $ s \in S$, to prove that 
$$\mathcal{A} \times_S \mathrm{Spec} \, \mathcal{O}_{S,s}$$ 
is the N\'eron model of its generic fiber. 

Recall the standard fact that, since $f$ is flat and of finite type, the formation of the smooth locus commutes with arbitrary basechange.

We put $ R = \mathcal{O}_{S,s} $ and denote by $R' $ the $\mathfrak{m}_R$-adic completion of $R$. Combining Proposition \ref{prop-Kulikov} and Lemma \ref{lemma-canonical}, we find that the pullback $ \mathcal{A} \times_S R' $ is a N\'eron model. However, it is also the pullback of $ \mathcal{A} \times_S R $ via the completion map $ R \to R' $. Thus, as we have explained in \ref{subsubsec-descent}, it then follows by descent that $ \mathcal{A} \times_S R $ is a N\'eron model.

\end{proof}

\section{Potential density on Calabi-Yau threefolds}
\label{sec: potential}

Here we collect some general results regarding potential density of rational points on abelian Calabi-Yau threefolds. We assume that our ground field $k$ is a field of characteristic zero.

\begin{lemma}
\label{lemma: specialization}
Assume that $S$ is a geometrically irreducible quasi-projective variety over $k$.
Let $f \colon \mathcal X \rightarrow S$ be a projective and faithfully flat morphism such that 
\begin{enumerate}
\item $\mathcal X$ is geometrically irreducible,
\item the generic fiber over the function field $k(S)$ is geometrically irreducible. 
\end{enumerate}
We denote the set of sections by $\mathrm{Sec}(\mathcal X/S)$ and consider a subset $\mathcal S \subset \mathrm{Sec}(\mathcal X/S)$.
Suppose that for some $s \in S(k)$, the image of $\mathcal S$ via the specialization map at $s$
\[
\mathcal S \subset \mathrm{Sec}(\mathcal X/S) \rightarrow \mathcal X_s(k)
\]
is Zariski dense in $\mathcal X_s(\overline{k})$.
Then the set of sections in $\mathcal S$ is Zariski dense in $\mathcal X$.
\end{lemma}
\begin{proof}
If $S$ is a point, the statement is obvious, so we can assume that $ \mathrm{dim}(S) > 0 $. We first prove our assertion when $S$ is an integral regular curve over $k$.
Let $\mathcal Z$ be the closure of the union of sections of $\mathcal S$ in $\mathcal X$.
Suppose that $ \mathcal Z$ is not equal to $\mathcal X$.
Let $Z$ denote one of the irreducible components of $\mathcal Z$. Then $Z$ necessarily maps dominantly to $S$. 
Since $S$ is an integral regular curve, the irreducible component $Z$ is flat over $S$. This implies that the intersection of $Z$ and $\mathcal X_s$ is a proper closed subset in $\mathcal X_s$. This contradicts with the assumption, because the specialization of any section in $\mathcal S$ must be contained in the intersection of one of irreducible components $Z$ with  $\mathcal X_s$.

Suppose that $S$ is a geometrically irreducible variety over $k$.
Let $\mathcal Z$ be the Zariski closure of the union of the elements of $\mathcal S$ in $\mathcal X$.
Since $S$ is quasi-projective, by a Bertini type argument, the union of all geometrically irreducible curves $C$ defined over $k$ and passing through $s$ forms a dense set in $S$. Let $C$ be such a curve, and denote by $\tilde{C}$ the normalization of $C$.
Then the result for integral regular curves shows that the (image of) the family $\mathcal X \times_S {\tilde{C}}$ is contained in $\mathcal Z$ for any $C$. This shows that $\mathcal Z = \mathcal X$.
\end{proof}

\begin{thm}
\label{thm: doublefibration}
Let $X$ be an abelian Calabi-Yau threefold defined over a number field $k$.
Suppose that $X$ has an effective movable non-big divisor $D$ such that the abelian fibration on $X$ does not factor through the Iitaka fibration of $D$ birationally. Then $X$ satisfies the potential density property for the Zariski topology.
\end{thm}
\begin{proof}
We denote the abelian fibration by $f \colon X\rightarrow \mathbb P^1$.
We apply $D$-MMP over $\overline{\mathbb Q}$ and obtain a sequence of $D$-flops $\phi \colon X \dashrightarrow \tilde{X}$ where the strict transform $\tilde{D}$ is nef. After taking some finite extension, we may assume that the result of MMP is defined over the ground field. Then $\tilde{X}$ is a smooth Calabi-Yau threefold and it follows from Theorem 1.1 page 99 of \cite{KMM94} (giving a proof of the log abundance conjecture in dimension 3, see \cite[Conjecture 3.12]{KM98} for the general log abundance conjecture) that $\tilde{D}$ is semi-ample. We consider its semi-ample fibration $g \colon \tilde{X} \rightarrow B$. We denote a general fiber of $g$ by $\tilde{Y}$ and its strict transform on $X$ by $Y$. Then $\tilde{Y}$ is either an elliptic curve, an abelian surface, or a K3 surface. When $\tilde{Y}$ is an elliptic curve or an abelian surface, then its potential density follows from \cite[Proposition 3.1]{HT00}. Suppose that $\tilde{Y}$ is a K3 surface. We denote the minimal resolution of $Y$ by $Y'$. Since the minimal model for a K3 surface is unique, the minimal model of $Y'$ is $\tilde{Y}$. Then since $Y$ admits a fibration to $\mathbb P^1$ via $f$, it is equipped with a semiample divisor of Iitaka dimension $1$. This implies that the K3 surface $\tilde{Y}$ admits a nef divisor of Iitaka dimension $1$ and it must be semi-ample. Thus $\tilde{Y}$ comes with fibrations, and the only fibrations on K3 surfaces are elliptic fibrations. The potential density for elliptic K3 surfaces is proved by the first author and Tschinkel (\cite{BoTs00}).

Pick a general point $t\in \mathbb P^1(k)$ and consider its smooth fiber $A_t = f^{-1}(t)$. This is an abelian surface. It follows from \cite[Proposition 3.1]{HaTs00} that after taking a finite extension of $k$, there exists a point $P \in A_t(k)$ such that the subgroup $\mathbb ZP$ is Zariski dense in $A_t$.
After replacing $P$ by a multiple of $P$, we may assume that $P$ is not contained in the indeterminacy of $\phi \colon X \dashrightarrow \tilde{X}$ and the fiber $\tilde{Y}$ passing through $P$ is smooth. We fix this $Y$.
After taking a finite extension if necessary, we may assume that the set of rational points on $Y$ is Zariski dense. By the assumption we made, $Y$ is not contained in $A_t$.
We denote by $V$ the fiber product $Y \times_{\mathbb P^1} X$. Being a pullback of $X$, $V$ admits an abelian surface fibration over $Y$ (with the zero section $\sigma$). 
By construction, the image of the specialization map
\[
\mathrm{Sec}(V/Y) \rightarrow V_P(k)
\]
is Zariski dense. Now the Zariski closure of the set of rational points on $V$ contains the set of sections because the set of rational points on $Y$ is Zariski dense. It follows from Lemma~\ref{lemma: specialization} that the set of rational points on $V$ is Zariski dense. Since $V$ is dominant to $X$, the set of rational points on $X$ is also Zariski dense.
\end{proof}

\section{The double covers of $\mathbb P^3$ ramified along octic surfaces (type $(1,4)$)}
\label{(1,4)}
In the remaining sections, we turn to examples of abelian Calabi-Yau threefolds introduced in \cite{GrPo01}. 
We assume that our ground field $k$ is an algebraically closed field of characteristic zero unless otherwise specified.
First we discuss abelian Calabi-Yau threefolds fibered by abelian surfaces of polarization type $(1,4)$ introduced in \cite[Section 2]{GrPo01}. See also \cite{BiLa04} pages 300-304.
Let $H_4$ be the Heisenberg group action on $\mathbb A^4$ generated by
\[
\sigma : x_i \mapsto x_{i-1}, \quad \tau : x_i \mapsto \zeta_4^{-i} x_i
\]
where $\zeta_4$ is a $4$-th root of unity.
We now change coordinates to
\[
z_0 = x_0 + x_1, \quad z_2 = x_3 + x_1,
\]
\[
z_1 = x_0-x_2, \quad z_3 = x_3 - x_1.
\]
We consider the following $3\times 3$ matrix:
\[
N_0 =
\begin{pmatrix}

 (z_0^4z_1^4 + z_2^4z_3^4) & (z_0^2z_1^2 + z_2^2z_3^2)(-z_0^2z_2^2 + z_1^2z_3^2) & (z_0^2z_1^2 - z_2^2z_3^2)(z_0^2z_3^2-z_1^2z_2^2)\\
 (z_0^2z_1^2+ z_2^2z_3^2)(-z_0^2z_2^2 + z_1^2z_3^2) & (z_0^4z_2^4 + z_1^4z_3^4) & (z_0^2z_2^2 + z_1^2z_3^2)(z_0^2z_3^2 + z_1^2z_2^2)\\
 (z_0^2z_1^2 - z_2^2z_3^2)(z_0^2z_3^2 -z_1^2z_2^2) & (z_0^2z_2^2 + z_1^2z_3^2)(z_0^2z_3^2 + z_1^2z_2^2 )& (z_0^4z_3^4 + z_1^4z_2^4)
\end{pmatrix},
\]
which is used to define by blocks the $4\times4$ matrix:
\[
N =
\begin{pmatrix}
z_0^2z_1^2z_2^2z_3^2 & 0_{1,3} \\
0_{3,1}& N_0
\end{pmatrix},
\]
where $0_{n,m}$ stands for an $n\times m$ matrix with zero entries everywhere.
For each $\lambda = (\lambda_0: \lambda_1:\lambda_2:\lambda_3) \in \mathbb P^3$, we define the following octic surface
\[
\overline{A}_\lambda : f_\lambda (z_0:z_1:z_2:z_3)= \lambda N^t\!\lambda = 0.
\]
For a general $\lambda \in \mathbb P^3$, the surface $\overline{A}_\lambda$ is smooth outside of 
$\{ z_0z_1z_2z_3=0 \}$. Set theoretically, the intersection of $\overline{A}_\lambda$ and $\{z_i = 0\}$ is a quartic curve with three ordinary double points at coordinate vertices, and this is a double curve on $\overline{A}_\lambda$. Its minimal resolution $A_\lambda$ is an abelian surface with a polarization of type $(1,4)$. 

Fix a line $l \subset \mathbb P^3$ and let $M$ be a $2\times 4$ matrix generating $l$
and we consider the following determinant:
\[
g_l = \det(M N {}^t\!M).
\]
This degree $16$ polynomial is divisible by $z_0^2z_1^2z_2^2z_3^2$ and we denote by $B_l$ the octic surface defined by the quotient of these polynomials:
\[
B_l : g_l/(z_0^2z_1^2z_2^2z_3^2) = 0.
\]
For a general $l$, the surface $B_l$ has of $148$ ordinary double points, which can be divided into three types: (A) $128$ are contained in the smooth locus of $\overline{A}_\lambda$ for any $\lambda \in l$; (B) $16$ others are contained in the double locus of $\overline{A}_\lambda$ for any $\lambda \in l$; (C) the remaining $4$ singularities are coordinate vertices.
We denote the double cover of $\mathbb P^3$ ramified along $B_l$ by $V_{4, l}$.
We also consider the following threefold:
\[
X_l = \{ (z, \lambda) \in \mathbb P^3 \times l \mid f_\lambda(z) = 0\}
\]
and we denote its normalization by $V^1_{4,l}$.
The following theorem is proved in \cite[Theorem 2.2]{GrPo01}:
\begin{thm}
The threefold $V^1_{4,l}$ is a small resolution of $V_{4,l}$
and it is a smooth abelian Calabi-Yau threefold fibered by $A_\lambda, \lambda \in l$.
\end{thm}
Exceptional curves above singularities of type (A) are sections of the fibration $V^1_{4,l}\rightarrow l$.
Curves above singularities of type (B) are multi-sections of degree $2$ tangent to $A_\lambda$, and exceptional curves above singularities of type (C) are multi-sections of degree $4$ tangent to $A_\lambda$ at a point.

We denote the pullback of the hyperplane class on $\mathbb P^3$ by $H$ and the class of abelian surfaces by $A$. On $V^1_{4,l}$, we have the following intersection numbers:
\[
H^3 = 2, \quad H^2A = 8, \quad A^2 = 0.
\]
After applying flops to $148$ exceptional curves, we obtain a birational model $V^2_{4,l}$ and on this model, the intersection numbers are given by
\[
H^3 = 2, \quad H^2A = 8, \quad HA^2 = 0, \quad A^3 = -512.
\]
On this model $8H - A$ is nef and it defines another fibration by abelian surfaces of type $(1,4)$.
To see this, one may note that there is a birational involution $\iota :V^1_{4,l} \dashrightarrow V^1_{4,l}$ and we have $8H - A \sim \iota(A)$. As a corollary we have:
\begin{cor}
Let $k$ be a number field and suppose that a line $l \subset \mathbb P^3$ is defined over $k$.
Then $V^1_{4,l}$ satisfies the potential density property for the Zariski topology.
\end{cor}
\begin{proof}
This follows from Theorem~\ref{thm: doublefibration} as $V^1_{4,l}$ admits a double fibration.
\end{proof}

\begin{rmk}
As shown in \cite[Theorem 2.2]{GrPo01}, the Calabi-Yau threefold $V^1_{4,l}$ has Picard rank $8$. It would be interesting to compute its Mordell-Weil group $\mathrm{MW}(A_\eta)$ where $\eta$ is the generic point of $l$ and see whether the rank is positive or not.
\end{rmk}

\section{Horrocks-Mumford quintic threefolds (type $(1,5)$)}\label{Horrocks}
In this section, we discuss Horrocks-Mumford quintic threefolds and their arithmetic properties.
We recall some basic facts about the Horrocks-Mumford bundle in \cite{HoMu73}, \cite{BHM87}, \cite{Aur89}, and \cite{Bor91}.
Let $e_i \, (i \in \mathbb Z_5)$ be a basis for the $5$-dimensional vector space $V_5$ over $k$.
We can think of them as homogeneous coordinates for $\mathbb P^4$.
The Heisenberg group $H_5 \subset \mathrm{SL}_5(k)$ is generated by two elements:
\[
\sigma \colon e_i \mapsto e_{i+1}, \quad \tau \colon e_i \mapsto \zeta_5^i e_i,
\]
where $\zeta_5$ is a primitive $5$-th root of unity.
Its center is generated by $\zeta I$ and its quotient by the center is isomorphic to $G := \mathbb Z_5 \times \mathbb Z_5$.

The Horrocks-Mumford bundle is a rank $2$ vector bundle $\mathcal F$ on $\mathbb P^4$, with general global sections $s \in H^0(\mathcal F)$ defining abelian surfaces $A_s \subset \mathbb P^4$ of degree $10$. Their polarization type is $(1,5)$.
The Heisenberg group $H_5$ is acting on $\mathcal F$ and
the group $G$ is acting on $A_s$ by translations that preserve the hyperplane class.
The abelian surfaces $A_s$ are also invariant under the involution:
\[
\iota \colon e_k \mapsto e_{-k}.
\]
We have $h^0(\mathcal F) = 4$ and $\wedge^2H^0(\mathcal F) = H^0(\wedge^2 \mathcal F)^{H_5} = H^0(\mathcal O(5))^{H_5}$. The $H_5$-invariant quintics define a linear system whose base locus is the union of $25$ lines:
\[
\cup_{i,j} L_{i,j} = \cup \sigma^i\tau^j \{x_0=x_1+x_4 = x_2+x_3 =0\}.
\]
Any member of this linear system is called a Horrocks-Mumford quintic, and its linear system defines a rational map $\xi \colon \mathbb P^4 \dashrightarrow \Omega \subset \mathbb P^5$ whose image is a quadric hypersurface $\Omega$ which can be identified with the Grassmanian $\mathrm{Gr}(2, H^0(\mathcal F))$.
Then, for general $x \in \mathbb P^4$, the map $\xi$ is given by
\[
x \mapsto \text{the pencil of sections vanishing at $x$},
\]
and it is generically finite of degree $100$.

It is shown in \cite[Section 5]{HoMu73} that the following fivefold is smooth:
\[
Z := \{(x, s) \in \mathbb P^4 \times \mathbb P(H^0(\mathcal F)) : x \in A_s\}.
\]
By Bertini's theorem, for any general pencil $\ell \subset \mathbb{P}(H^0(\mathcal F))$
the resulting threefold 
$$\tilde{X} = Z \cap (\mathbb P^4 \times \ell)$$ 
is smooth. This is a smooth abelian Calabi-Yau threefold, called a Horrocks-Mumford Calabi-Yau threefold.
Let $s_1, s_2$ be a basis for $\ell$.
Then $\tilde{X}$ is a smooth resolution of the Horrocks-Mumford quintic $X$ defined by $s_1\wedge s_2 =0$, and whose singular locus $X_{\mathrm{sing}}$ consists of $100$ nodes. The Betti numbers of $\tilde{X}$ are given by
\[
b_1(\tilde{X}) = 0, \, b_2(\tilde{X}) = 4, \, b_3(\tilde{X})=10.
\]
\begin{lemma}
\label{lemma: simplyconn}
Suppose that our ground field is $\mathbb C$.
Then any Horrocks-Mumford Calabi-Yau threefold is an abelian Calabi-Yau threefold in the strict sense, i.e., 
$\tilde{X}$ is simply connected.
\end{lemma}

\begin{proof}
Let $ \pi \colon \tilde{X} \to \ell= \mathbb P^1$ be the threefold constructed above. We claim that $\pi \colon \tilde{X} \to \mathbb P^1$  is an abelian Calabi-Yau threefold. Indeed, $h^1(\tilde{X}, \mathcal O) =0$ because of Lefschetz hyperplane theorem and the birational invariance of $h^1(\mathcal O)$. Also we have $\tilde{X}(S) \neq \emptyset $ because the $\mathbb P^1$'s over the $100$ nodes form sections. Finally for a general $s \in \ell = \mathbb P^1$, the fiber $A_s$ is an abelian surface of polarization type $(1, 5)$.

Let $X$ the corresponding Horrocks-Mumford quintic in $\mathbb P^4$.
It follows from the Lefschetz hyperplane theorem that $\pi_1(X) = 0$.
Each node in $X_{\mathrm{sing}}$ is analytically isomorphic to $\{xw-yz = 0\} \subset \mathbb A^4$,
so, in particular, the local fundamental group $\pi_1^{\mathrm{loc}}(X, x)$ of each node $x$ is zero.
It follows from the Seifert-van Kampen theorem that $\pi_1(X \setminus X_{\mathrm{sing}})=0$.
Let $L_i$ be the exceptional curves of the small resolution $\tilde{X} \rightarrow X$.
We have $\pi_1(\tilde{X} \setminus \cup L_i) = 0$.
Again, it follows from the Seifert-van Kampen theorem that $\pi_1(\tilde{X}) = 0$.
\end{proof}

\subsection{The (birational) automorphism group}
The possible degenerate surfaces one can obtain in the Horrocks-Mumford construction have been classified in \cite{BHM87}. We list the possibilities as follows:

\begin{theorem}{\cite[Theorem (0.1)]{BHM87}}
Each Horrocks-Mumford surface $A_s$ is one of the following:
\begin{enumerate}
\item a smooth abelian surface,

\item a translation scroll associated to a normal elliptic quintic curve,

\item the tangent scroll of a normal elliptic quintic curve,

\item a quintic elliptic scroll carrying a multiplicity-$2$ structure,

\item the union of five smooth quadric surfaces,

\item the union of five planes carrying a multiplicity-$2$ structure.
\end{enumerate}
\end{theorem}

The degenerate fibers of type $(4)$ and $(6)$ in the above list cannot occur. Indeed, a section must intersect any closed fiber $\tilde{X}_s$ in a smooth point.
Also, if a fiber of type (5) occurs, then $\tilde{X}$ must be singular. Indeed, from the equation of the union of five smooth quadric surfaces in \cite[Section 3]{BHM87}, one can observe that the intersection of two quadrics in this singular fiber is the union of a line and a point.
This is impossible in a smooth threefold.

Let us denote by $\mathcal{A}$ the $\pi$-smooth locus of $\tilde{X}$. We observed in Theorem~\ref{theorem: Neronmodel} that $ \mathcal{A} \to S $ is a N\'eron model. This means that the smooth locus of degenerate fibers of type $(2)$, $(3)$ or $(5)$ acquires a group structure. 
In fact, $(2)$ yields semi-abelian surfaces of torus rank $1$.

Now we consider $ X \to S $, and, on restriction to its smooth locus, the N\'eron model $ \mathcal{A} \to S $. Any section $ P \in X(S) $ restricts to a section of $\mathcal{A}$ (since it must pass through the smooth locus), which we also denote $P$. It induces a morphism $ t_P \colon \mathcal{A} \to \mathcal{A} $ relative to $S$, which is translation in each fiber $\mathcal{A}_s$. On $X$, this yields a birational map $ f_P \colon X \dashrightarrow X $. 

\begin{prop}
The birational map $f_P$ is a regular automorphism of $X$.
\end{prop}
\begin{proof}

It follows from Lemma~\ref{lemma:flops} that the indeterminacy of $f_P$ is the union of rational curves.
On the other hand, since $f_P$ is regular on $\mathcal{A}$, the possible indeterminacy is along the singular loci of singular fibers. 
The only possible singular fibers are (2) or (3), but these are singular along the normal quintic elliptic curve. Thus our assertion follows.
\end{proof}

The subgroup of the regular automorphism group given by translations of the abelian fibrations is computed in \cite[Corollary 3.6]{Bor91}. 
Indeed, Borcea computed this group by constructing generators of this group and computing the nef cone of divisors.
\begin{cor}
Let $\eta \in S$ be the generic point and $\mathcal{A}_\eta$ be the generic fiber.
Then the Mordell-Weil group $\mathrm{MW}(\mathcal{A}_\eta) = \tilde{X}(S)$ is isomorphic to $\mathbb Z^2 \times \mathbb Z_5 \times \mathbb Z_5$.
\end{cor}
\subsection{The density of rational points}

Let $M_y(x) =[x_{i+j}y_{i-j}]_{0\leq i, j \leq 4}$, \textit{i.e.}
\[
M_y(x) =
 \begin{pmatrix}
 x_0y_0&x_1y_4 & x_2y_3&x_3y_2&x_4y_1\\
 x_1y_1&x_2y_0&x_3y_4&x_4y_3&x_0y_2\\
 x_2y_2 & x_3y_1 & x_4y_0&x_0y_4&x_1y_3\\
 x_3y_3& x_4y_2&x_0y_1&x_1y_0&x_2y_4\\
 x_4y_4&x_0y_3&x_1y_2&x_2y_1&x_3y_0
 \end{pmatrix}.
\]
Ross Moore showed in \cite{Mo85} that any Horrocks-Mumford quintic $X$ is determinantal and it is defined by
$\det (M_y(x))=0$ where $y \in \mathbb P^4$ is any node of $X$, see also \cite[p.~26]{Bor91}.
We can choose two nodes $y_1, y_2$ so that the $100$ nodes on $X$ are given by $\langle G, \iota \rangle \{y_1, y_2\}$. Note that $\langle G, \iota \rangle y_1$ consists of $50$ distinct points so $100$ points are formed by two orbits of $\langle G, \iota \rangle$. 
Let $N$ be a general $4\times 5$ matrix and define surfaces
\[
\Delta_i = \{x \in \mathbb P^4 \mid \mathrm{rank} (NM_{y_i}(x)) = 3\}.
\]
Let $H$ be the pullback of the hyperplane class on $\tilde{X}$ and $A$ a general fiber of the abelian fibration $\pi : \tilde{X} \rightarrow S$. 
It is shown in \cite{Aur89} that $H, A, \tilde{\Delta}_1, \tilde{\Delta}_2$ form a basis for $\mathrm{Pic}(\tilde{X})$ where $ \tilde{\Delta}_1, \tilde{\Delta}_2$ are the strict transformations of  $\Delta_1, \Delta_2$.

\begin{prop}
Let $\eta \in S$ be the generic point and $A_\eta$ the generic fiber.
Then $A_\eta$ is simple.
\end{prop}
\begin{proof}
Using \cite[Proposition 3.1]{Bor91}, one has
\[
(\beta_0A + \beta_1 H + \beta_2 \tilde{\Delta}_i)^3 = 5(\beta_1 + 2 \beta_2)(\beta_1^2 + 4 \beta_2^2 + 6\beta_0\beta_1 + 12\beta_0\beta_2 + 4 \beta_1\beta_2).
\]
Expanding polynomials in $\beta_0, \beta_1, \beta_2$, one sees $H^2\cdot A = 10, H\cdot \tilde{\Delta}_i \cdot A = 20, \tilde{\Delta}_i^2 \cdot A = 40$.
This implies that the intersection matrix of the restrictions of $H$ and $\tilde{\Delta}_i$ to $A_\eta$ is degenerate, hence these restrictions are linearly dependent.
In particular, the image of the restriction map $\mathrm{NS}(X) \rightarrow \mathrm{NS}(A_\eta)$ has rank one. However, if $A_\eta$ is not simple over the function field $k(\eta)$, then it contains two elliptic curves meeting with a positive intersection number. Thus the image $\mathrm{NS}(X) \rightarrow \mathrm{NS}(A_\eta)$ must have rank $\geq 2$. It is a contradiction, hence $A_\eta$ has to be simple over the function field $k(\eta)$. 
\end{proof}

\begin{cor}
The set $\tilde{X}(S)$ of sections is Zariski dense in $X$.
\end{cor}
\begin{proof}
The Zariski closure of $\tilde{X}(S)$ contains the closure of $A_\eta(k(\eta))$.
The Zariski closure of $A_\eta(k(\eta))$ in $A_\eta$ forms a sub-group scheme $\mathcal G$ of $A_\eta$,
and it cannot be a finite set because there is a section of infinite order. On the other hand, $A_\eta$ is simple so $\mathcal G$ cannot be $1$-dimensional. Thus the only possibility is that $\mathcal G = A_\eta$. This means that the Zariski closure of $A_\eta(k(\eta))$ in $\tilde{X}$ is $\tilde{X}$ because of the irreducibility of $\tilde{X}$. Our assertion follows.
\end{proof}

\begin{cor}
Suppose that our ground field is a number field $k$ containing a $5$-th root of unity and $y \in \mathbb P^4(k)$ defines a Horrocks-Mumford quintic $X$ defined by $\det(M_y(x))=0$.
Suppose that its small resolution $\tilde{X}$ is smooth so that $\tilde{X}$ is a Horrocks-Mumford Calabi-Yau threefold.
Then the set  $X(k)$ of rational points is Zariski dense in $X$.
\end{cor}
\begin{proof}
First observe that the pencil of sections passing through $y$ is defined over $k$, so the resulting Calabi-Yau threefold $\tilde{X}$ is also defined over $k$.
Next note that $50$ of the nodes of $X$ are given by $\langle G, \iota \rangle y$ which are defined over $k$ since $k$ contains a $5$-th root of unity.
So, in particular, $\tilde{X}$ contains at least $50$ sections defined over the ground field $k$.
As we have seen, the torsion part of $\mathrm{MW}(A_\eta)$ has order $25$,
so there must be a section of infinite order defined over $k$.
By virtue of Remark \ref{simplicity function fields}, the simplicity of $A_\eta$ implies that the closure of $\mathrm{MW}(A_{\eta})$ inside $A_{\eta}$ is $A_{\eta}$. As the closure of the generic fiber $A_{\eta}$ of $X$ is $X$, this concludes our assertion.
\end{proof}

\section{The complete intersections of two cubics (type $(1,6)$)}
\label{(1,6)}
In this section, we deal with abelian Calabi-Yau threefolds fibered by abelian surfaces of polarization type $(1,6)$ discussed in \cite[Section 4]{GrPo01}.

We consider the Heisenberg group $H_6 \subset \mathrm{GL}_6(k) = \mathrm{GL}(V_6)$ generated by two elements
\[
\sigma (x_i) = x_{i-1}, \quad \tau (x_i) = \zeta_6^{-i}x_i
\]
where $\zeta_6$ is a primitive $6$-th root of unity.
We consider $\mathbb P^5 = \mathbb (V_6)$ and the action of $H_6$ on $\mathbb P_5$.
We consider the subgroup $H' \subset H_6$ generated by $\sigma^2, \tau^2$ and look at $H^0(\mathbb P^5, \mathcal O(3))^{H'}$.
It has a basis $f_0, \cdots f_3, \sigma f_0, \cdots \sigma f_3$ where
\begin{align*}
&f_0 = x_0^3 + x_2^3 + x_4^3,\\
&f_1 = x_1^2x_4 + x_3^2x_0 + x_5^2x_2,\\
&f_2 = x_1x_2x_3+x_3x_4x_5+x_5x_0x_1,\\
&f_3 = x_0x_2x_4.
\end{align*}
The vector space $H^0(\mathcal O(3))^{H'}$ is a representation of $H_6$ which is the direct sum of four copies of an irreducible two dimensional representation:
\[
H^0(\mathcal O(3))^{H'} = \oplus_{i = 0}^3 \langle f_i, \sigma f_i \rangle.
\]
We denote this irreducible representation by $V_0$.
Let $W$ be a four dimensional vector space with a basis $e_0, \cdots e_3$ such that
$V_0 \otimes \langle e_i \rangle = \langle f_i, \sigma f_i \rangle$.

For a point $p\in \mathbb P (W)$ corresponding to a one dimensional vector space $T \subset W$,
let $V_{6,p}$ be the complete intersection of type $(3,3)$ defined by cubics in $V_0 \otimes T \subset H^0(\mathcal O(3))^{H'}$. It is easy to see that this complete intersection is invariant under the action of $H_6$.
\begin{thm}{\cite[Theorem 4.10]{GrPo01}}
For a general $p\in \mathbb P(W)$, we have
\begin{enumerate}
\item $V_{6,p}$ is an irreducible threefold with $72$ ordinary double points;
\item There is a small resolution $V^1_{6,p} \rightarrow V_{6,p}$ of the ordinary double points
such that $V^1_{6,p}$ is an abelian Calabi-Yau threefold fibered by abelian surfaces of polarization type $(1,6)$, and it has Picard rank $6$.
\end{enumerate}

\end{thm}

Following Lemma~\ref{lemma: simplyconn}, it is easy to verify the following proposition:
\begin{prop}
Suppose that the ground field is $\mathbb C$. Then $V^1_{6,p}$ is a Calabi-Yau threefold in the strict sense.
\end{prop}

The potential density property for $V_{6, p}$ is an easy consequence of Theorem~\ref{thm: doublefibration}:
\begin{cor}
Let $k$ be a number field and $p \in \mathbb P(W)(k)$ a general point.
Then $V_{6, y}$ satisfies the potential density property for the Zariski topology.
\end{cor}
\begin{proof}
Applying flops twice to $V^1_{6, y}$ yields a birational model that admits another abelian fibration with fibers of polarization type $(2,6)$ (See \cite[Remark 4.12]{GrPo01}). Thus, potential density of $V_{6, y}$ follows from Theorem~\ref{thm: doublefibration}.
\end{proof}

\begin{rmk}
It would be interesting to study the Mordell-Weil group of the generic fiber of $V^1_{6,p}$ and show that there is a section of infinite order.
\end{rmk}

\section{Pfaffian Calabi-Yau threefolds (type $(1,7)$)}\label{Pfaffian}

We consider the Heisenberg group $H_7 \subset \mathrm{GL}_7(k) = \mathrm{GL}(V_7)$ generated by two elements
\[
\sigma (x_i) = x_{i-1}, \quad \tau (x_i) = \zeta_7^{-i}x_i
\]
where $\zeta_7$ is a primitive $7$-th root of unity.
We also consider the involution
\[
\iota(x_i) = x_{-i}.
\]
We denote the positive and negative eigenspaces of the Heisenberg involution $\iota$ by $V_+$ and $V_-$, and we consider the following matrix
\[
M'_7(x, y)= \left(x_{\frac{(i+j)}{2}}y_{\frac{(i-j)}{2}}\right)_{i, j \in \mathbb Z_7}.
\]
For any parameter point $y = (0:y_1:y_2:y_3:-y_3:-y_2:-y_1) \in \mathbb P^2_- = \mathbb P(V_-)$, the matrix $M_7(x, y)$ is alternating. We denote the closed subscheme defined by $6 \times 6$ - Pfaffians of the alternating matrix $M'_7(x, y)$ by $V_{7,y} \subset \mathbb P^6$.

\begin{proposition}{\cite[Proposition 5.8]{GrPo01}}
Let $y \in \mathbb P^2_-$ be a general point. Then the following properties hold.
\begin{enumerate}
\item The threefold $V_{7, y}$ has $49$ ordinary double points as singularities which occur at the $H_7$-orbit of $y$. Moreover it contains a pencil of polarized abelian surfaces of type $(1,7)$.
\item There exists a small resolution $V^1_{7, y}\rightarrow V_{7, y}$ such that $V^1_{7, y}$ is a smooth Calabi-Yau threefold of Picard rank $2$ with an abelian fibration $\pi_1 \colon V^1_{7, y} \rightarrow \mathbb P^1$ whose fibers form a pencil of $(1,7)$-polarized abelian surfaces. Moreover, the $49$ projective lines over the ordinary double points form sections of this fibration.
\end{enumerate}
\end{proposition}

First we discuss potential density for $V_{7, y}$. This is an easy consequence of Theorem~\ref{thm: doublefibration}:
\begin{cor}
Let $k$ be a number field and $y \in \mathbb P^2_-(k)$ a general point.
Then $V_{7, y}$ satisfies the potential density property with respect to the Zariski topology.
\end{cor}
\begin{proof}
Applying flops twice to $V^1_{7, y}$ yields a birational model that admits another abelian fibration with fibers of polarization type $(1, 14)$ (See \cite[Remark 5.9]{GrPo01}). Thus, potential density of $V_{7, y}$ follows from Theorem~\ref{thm: doublefibration}.
\end{proof}

\begin{proposition}
\label{prop: Mordell-(1, 7)}
The Mordell-Weil group $\mathrm{MW}(\pi_1)$ is isomorphic to $\mathbb Z_7 \oplus \mathbb Z_7$.
\end{proposition}
\begin{proof}
First note that the Heisenberg group $H_7$ fixes the abelian surfaces of type $(1,7)$ and acts on them by translations. 

Let $s \in \mathrm{MW}(\pi_1)$ be a section of $\pi_1$ and consider its birational automorphism $f_s : V^1_{7, y} \dashrightarrow V^1_{7, y}$ given by the translation. Since the Picard rank of $V^1_{7, y}$ is $2$, this birational automorphism is an actual regular automorphism and acts on the Picard group of $V^1_{7, y}$ trivially. Thus the pullback of the hyperplane class to $V_{7, y}\subset \mathbb P^6$ is fixed by $f_s$, so $f_s$ acts on $V_{7, y} \subset \mathbb P^6$ linearly. However, $f_s$ fixes the abelian surfaces themselves and their polarizations, hence the action is given by $H_7/Z(H_7) = \mathbb Z_7 \oplus \mathbb Z_7$. Thus our assertion follows.
\end{proof}

\section{The $(2, 2, 2, 2)$-complete intersections (type $(1,8)$)}\label{2222}
We consider the Heisenberg group $H_8 \subset \mathrm{GL}_8(k) = \mathrm{GL}(V_8)$ generated by two elements
\[
\sigma (x_i) = x_{i-1}, \quad \tau (x_i) = \zeta_8^{-i}x_i
\]
where $\zeta_8$ is a primitive $8$-th root of the unity.
We also consider the involution
\[
\iota(x_i) = x_{-i}.
\]
We denote the positive and negative eigenspaces of the Heisenberg involution $\iota$ by $V_+$ and $V_-$. The Heisenberg group $H_8$ acts on $\mathbb P^7 = \mathbb P(V_8)$ and we consider the subgroup $H' \subset H_8$ generated by $\sigma^4, \tau^4$. We also consider the following matrix:
\[
M_4(x, y) = \left(x_{i+j}y_{i-j} + x_{i+j+4}y_{i-j+4}\right)_{0\leq i, j\leq 3}.
\]
The space of $H'$-invariant quadrics vanishing along the $H_8$-orbit of $y\in \mathbb P_-^2$ in 
$H^0(\mathcal O_{\mathbb P^7}(2))^{H'}$ is $4$ dimensional for a general $y$ and it is spanned by the $4\times 4$ - Pfaffians of the matrices $M_4(x, y), M_4(\sigma^4(x), y), M_4(\tau^4(x), y), M_4(\sigma^4\tau^4(x), y)$. We denote the $(2, 2, 2, 2)$-complete intersection of these quadrics by $V_{8, y}$. 
\begin{thm}{\cite[Theorem 6.5 and 6.9]{GrPo01}}
For a general $y\in \mathbb P_-^2$, $V_{8, y}$ is a singular Calabi-Yau threefold with exactly $64$ ordinary double points consisting of the $H_8$-orbit of $y$. It comes with a pencil of polarized abelian surfaces of type $(1, 8)$. One of its small resolutions comes with the abelian fibration $\pi_1: V^1_{8, y}\rightarrow \mathbb P^1$ of type $(1,8)$ and it is a smooth Calabi-Yau threefold with Picard rank $2$.
\end{thm}

Following Lemma~\ref{lemma: simplyconn}, it is easy to verify the following proposition:
\begin{prop}
Suppose that the ground field is $\mathbb C$. Then $V^1_{8,y}$ is a Calabi-Yau threefold in the strict sense.
\end{prop}

Moreover it comes with another abelian fibration after applying one flop (\cite[Proposition 6.14]{GrPo01}), so the potential density follows from Theorem~\ref{thm: doublefibration}.
\begin{cor}
Let $k$ be a number field and suppose that $y \in \mathbb P_-^2(k)$ is general. Then $V_{8, y}$ satisfies the potential density property with respect to the Zariski topology.
\end{cor}

The computation of the Mordell-Weil group of $\pi_1$ is similar to Proposition~\ref{prop: Mordell-(1, 7)}.

\begin{proposition}
\label{prop: Mordell-(1, 8)}
The Mordell-Weil group $\mathrm{MW}(\pi_1)$ is isomorphic to $\mathbb Z_8 \oplus \mathbb Z_8$.
\end{proposition}

\section{The intersections of two Grassmannians (type $(1,10)$)}\label{Grassmannians}
We consider the Heisenberg group $H_{10} \subset \mathrm{GL}_{10}(k) = \mathrm{GL}(V_{10})$ generated by two elements
\[
\sigma (x_i) = x_{i-1}, \quad \tau (x_i) = \zeta_{10}^{-i}x_i
\]
where $\zeta_{10}$ is a primitive $10$-th root of the unity.
We also consider the involution
\[
\iota(x_i) = x_{-i}.
\]
We denote the positive and negative eigenspaces of the Heisenberg involution $\iota$ by $V_+$ and $V_-$. We also consider the following matrix:
\[
M_5(x, y) = \left(x_{i+j}y_{i-j} + x_{i+j+5}y_{i-j+5}\right)_{0\leq i, j\leq 4}.
\]
For a general point $y \in \mathbb P^3_-$, the $4\times 4$ - Pfaffians of $M_{5}(x,y)$ cut out a variety $G_y \subset \mathbb P^{9}$ which can  be identified with a Pl\"ucker embedding of $\mathrm{Gr}(2, 5)$. We define $V_{10, y}$ to be $G_y \cap \tau^5(G_y)$.
\begin{thm}{\cite[Theorem 7.4 and Remark 7.5]{GrPo01}}
For a general $y\in \mathbb P_-^3$, we have that $V_{10, y}$ is a $H_{10}$-invariant singular Calabi-Yau threefold with exactly $100$ ordinary double points. It comes with a pencil of polarized abelian surfaces of type $(1, 10)$. One of its small resolutions comes with the abelian fibration $\pi_1: V^1_{10, y}\rightarrow \mathbb P^1$ of type $(1,10)$ and it is a smooth Calabi-Yau threefold of Picard rank $2$.
\end{thm}

Moreover, it comes with another abelian fibration after applying flops twice (\cite[Remark 7.5]{GrPo01}), so the potential density follows from Theorem~\ref{thm: doublefibration}.
\begin{cor}
Let $k$ be a number field and suppose that $y \in \mathbb P_-^3(k)$ is general. Then $V_{10, y}$ satisfies the potential density property with respect to the Zariski topology.
\end{cor}

The computation of the Mordell-Weil group of $\pi_1$ is similar to Proposition~\ref{prop: Mordell-(1, 7)}.

\begin{proposition}
\label{prop: Mordell-(1, §0)}
The Mordell-Weil group $\mathrm{MW}(\pi_1)$ is isomorphic to $\mathbb Z_{10} \oplus \mathbb Z_{10}$.
\end{proposition}


\bibliographystyle{alpha}
\bibliography{AbelianCY}

\begin{thebibliography}{BCHM10}

\bibitem[Aur89]{Aur89}
Alf~Bj{\o}rn Aure.
\newblock Surfaces on quintic threefolds associated to the {H}orrocks-{M}umford
  bundle.
\newblock In {\em Arithmetic of complex manifolds ({E}rlangen, 1988)}, volume
  1399 of {\em Lecture Notes in Math.}, pages 1--9. Springer, Berlin, 1989.

\bibitem[BCHM10]{BCHM}
Caucher Birkar, Paolo Cascini, Christopher~D. Hacon, and James McKernan.
\newblock Existence of minimal models for varieties of log general type.
\newblock {\em J. Amer. Math. Soc.}, 23(2):405--468, 2010.

\bibitem[BHM87]{BHM87}
Wolf Barth, Klaus Hulek, and Ross Moore.
\newblock Degenerations of {H}orrocks-{M}umford surfaces.
\newblock {\em Math. Ann.}, 277(4):735--755, 1987.

\bibitem[BL04]{BiLa04}
Christina Birkenhake and Herbert Lange.
\newblock {\em Complex abelian varieties}, volume 302 of {\em Grundlehren der
  Mathematischen Wissenschaften [Fundamental Principles of Mathematical
  Sciences]}.
\newblock Springer-Verlag, Berlin, second edition, 2004.

\bibitem[BLR90]{neron}
Siegfried Bosch, Werner L{\"u}tkebohmert, and Michel Raynaud.
\newblock {\em N\'eron models}, volume~21 of {\em Ergebnisse der Mathematik und
  ihrer Grenzgebiete (3) [Results in Mathematics and Related Areas (3)]}.
\newblock Springer-Verlag, Berlin, 1990.

\bibitem[Bor91]{Bor91}
Ciprian Borcea.
\newblock On desingularized {H}orrocks-{M}umford quintics.
\newblock {\em J. Reine Angew. Math.}, 421:23--41, 1991.

\bibitem[BT98]{BoTs98}
F.~A. Bogomolov and Yu. Tschinkel.
\newblock Density of rational points on {E}nriques surfaces.
\newblock {\em Math. Res. Lett.}, 5(5):623--628, 1998.

\bibitem[BT99]{BoTs99}
F.~A. Bogomolov and Yu. Tschinkel.
\newblock On the density of rational points on elliptic fibrations.
\newblock {\em J. Reine Angew. Math.}, 511:87--93, 1999.

\bibitem[BT00]{BoTs00}
F.~A. Bogomolov and Yu. Tschinkel.
\newblock Density of rational points on elliptic {$K3$} surfaces.
\newblock {\em Asian J. Math.}, 4(2):351--368, 2000.

\bibitem[Con06]{Con06}
Brian Conrad.
\newblock Chow's {$K/k$}-image and {$K/k$}-trace, and the {L}ang-{N}\'eron
  theorem.
\newblock {\em Enseign. Math. (2)}, 52(1-2):37--108, 2006.

\bibitem[GP98]{GrPo98}
Mark Gross and Sorin Popescu.
\newblock Equations of {$(1,d)$}-polarized abelian surfaces.
\newblock {\em Math. Ann.}, 310(2):333--377, 1998.

\bibitem[GP01]{GrPo01}
Mark Gross and Sorin Popescu.
\newblock Calabi-{Y}au threefolds and moduli of abeian surfaces. {I}.
\newblock {\em Compositio Math.}, 127(2):169--228, 2001.

\bibitem[HM73]{HoMu73}
G.~Horrocks and D.~Mumford.
\newblock A rank {$2$} vector bundle on {${\bf P}^{4}$} with {$15,000$}\
  symmetries.
\newblock {\em Topology}, 12:63--81, 1973.

\bibitem[HN09]{HaNi}
Lars~Halvard Halle and Johannes Nicaise.
\newblock Motivic zeta functions of abelian varieties, and the monodromy
  conjecture, 2009.
\newblock {\tt arXiv: 0902.3755v3}.

\bibitem[HT00a]{HaTs00}
Joe Harris and Yuri Tschinkel.
\newblock Rational points on quartics.
\newblock {\em Duke Math. J.}, 104(3):477--500, 2000.

\bibitem[HT00b]{HT00}
Brendan Hassett and Yuri Tschinkel.
\newblock Abelian fibrations and rational points on symmetric products.
\newblock {\em Internat. J. Math.}, 11(9):1163--1176, 2000.

\bibitem[Kaw88]{Kaw88}
Yujiro Kawamata.
\newblock Crepant blowing-up of {$3$}-dimensional canonical singularities and
  its application to degenerations of surfaces.
\newblock {\em Ann. of Math. (2)}, 127(1):93--163, 1988.

\bibitem[KM98]{KM98}
J{\'a}nos Koll{\'a}r and Shigefumi Mori.
\newblock {\em Birational geometry of algebraic varieties}, volume 134 of {\em
  Cambridge Tracts in Mathematics}.
\newblock Cambridge University Press, Cambridge, 1998.
\newblock With the collaboration of C. H. Clemens and A. Corti, Translated from
  the 1998 Japanese original.

\bibitem[KMM94]{KMM94}
Sean Keel, Kenji Matsuki, and James McKernan.
\newblock Log abundance theorem for threefolds.
\newblock {\em Duke Math. J.}, 75(1):99--119, 1994.

\bibitem[Kol89]{Kol89}
J{\'a}nos Koll{\'a}r.
\newblock Flops.
\newblock {\em Nagoya Math. J.}, 113:15--36, 1989.

\bibitem[Liu02]{liu}
Qing Liu.
\newblock {\em Algebraic geometry and arithmetic curves}, volume~6 of {\em
  Oxford Graduate Texts in Mathematics}.
\newblock Oxford University Press, Oxford, 2002.
\newblock Translated from the French by Reinie Ern{\'e}, Oxford Science
  Publications.

\bibitem[LMvL10]{LMvL10}
Adam Logan, David McKinnon, and Ronald van Luijk.
\newblock Density of rational points on diagonal quartic surfaces.
\newblock {\em Algebra Number Theory}, 4(1):1--20, 2010.

\bibitem[LN59]{LaNe59}
S.~Lang and A.~N{\'e}ron.
\newblock Rational points of abelian varieties over function fields.
\newblock {\em Amer. J. Math.}, 81:95--118, 1959.

\bibitem[Moo85]{Mo85}
Roger Moore.
\newblock Heisenberg invariant quintic 3-folds and sections of the
  {H}orrocks-{M}umford bundle, 1985.
\newblock Research Report, The Australian National University, Mathematical
  Sciences Research Centre, No. 33, 55 pages,.

\bibitem[Ogu14]{Ogu14}
Keiji Oguiso.
\newblock Automorphism groups of {C}alabi-{Y}au manifolds of {P}icard number 2.
\newblock {\em J. Algebraic Geom.}, 23(4):775--795, 2014.

\bibitem[SD68]{Sw68}
H.~P.~F. Swinnerton-Dyer.
\newblock {$A^{4}+B^{4}=C^{4}+D^{4}$} revisited.
\newblock {\em J. London Math. Soc.}, 43:149--151, 1968.

\bibitem[SD13]{Sw13}
Peter Swinnerton-Dyer.
\newblock Density of rational points on certain surfaces.
\newblock {\em Algebra Number Theory}, 7(4):835--851, 2013.

\bibitem[Tsc06]{Tsc06}
Yuri Tschinkel.
\newblock Geometry over nonclosed fields, 2006.
\newblock ICM talk,
  http://www.cims.nyu.edu/~tschinke/papers/yuri/06icm/tschinkel.pdf,.

\end{thebibliography}

\end{document}